\documentclass[12pt, reqno]{amsart}

\usepackage[T1]{fontenc}
\usepackage[utf8]{inputenc}

\usepackage{microtype}
\usepackage{lmodern}

\usepackage{pdflscape}

\usepackage[a4paper,portrait, margin={1in}]{geometry}

\usepackage{amssymb,amsmath}
\usepackage{paralist}

\usepackage[english]{babel} 
\usepackage[autostyle]{csquotes} 

\usepackage{breqn}

\usepackage[]{xcolor}
\usepackage[colorlinks=true, citecolor=blue, urlcolor=violet, breaklinks=true]{hyperref}

\newtheorem{theorem}{Theorem}[section]
\newtheorem{proposition}[theorem]{Proposition}

\newtheorem{corollary}[theorem]{Corollary}
\newtheorem{problem}[theorem]{Problem}
\newtheorem{conjecture}[theorem]{Conjecture}

\theoremstyle{definition}
\newtheorem{definition}[theorem]{Definition}
\newtheorem{example}[theorem]{Example}
\newtheorem{remark}[theorem]{Remark}

\newcommand{\NN}{\mathbb{N}}
\newcommand{\ZZ}{\mathbb{Z}}

\newcommand{\QQ}{\mathbb{Q}}
\newcommand{\kk}{\Bbbk}

\DeclareMathOperator{\IN}{In}
\DeclareMathOperator{\Quot}{Quot}

\newcommand{\set}[1]{\{#1\}}
\newcommand{\with}{\ \vrule\ }

\newcommand{\LH}{\mathbf{L}_{n}}

\newcommand{\lh}{\ell}
\newcommand{\LHP}{{\mathbf{P}_{\mathrm{LH}}}}
\newcommand{\PP}{\mathbf{P}}

\newcommand{\mino}{\mathcal{E}}
\newcommand{\sm}[3]{\Delta^{#1}_{#2}(#3)}

\newcommand{\cvec}[1]{\begin{pmatrix} #1 \end{pmatrix}}

\begin{document} 

\title[]{The Lecture Hall Cone as a toric deformation}

\author{Lukas Katth\"an}
\address{Institut f\"ur Mathematik, Goethe-Universit\"at Frankfurt, Germany}
\email{katthaen@math.uni-frankfurt.de}

\date{\today}


\begin{abstract}
	The Lecture Hall cone is a simplicial cone whose lattice points naturally correspond to Lecture Hall partitions.
	The celebrated Lecture Hall Theorem of Bousquet-M\'elou and Eriksson states that a particular specialization of its multivariate Ehrhart series factors in a very nice and unexpected way.
	Over the years, several proofs of this result have been found, but it is still not considered to be well-understood from a geometric perspective.
	
	In this note we propose two conjectures which aim at clarifying this result.
	Our main conjecture is that the Ehrhart ring of the Lecture Hall cone is actually an initial subalgebra $A_n$ of a certain subalgebra of a polynomial ring, which is itself isomorphic to a polynomial ring.
	As passing to initial subalgebras does not affect the Hilbert function, this explains the observed factorization.
	We give a recursive definition of certain Laurent polynomials, which generate the algebra $A_n$. Our second conjecture is that these Laurent polynomials are in fact polynomials.
	
	We computationally verified that both conjectures hold for Lecture Hall partitions of length at most 12.
\end{abstract}

\maketitle

\section{Introduction}

\noindent A \emph{Lecture Hall partition}\footnote{Our indexing of the entries of $\lambda$ is reversed with respect to the usual convention.} is a finite sequence $\lambda = (\lambda_1, \lambda_2, \dotsc, \lambda_n) \in \ZZ^n$ satisfying
\[\frac{\lambda_1}{n} \geq \frac{\lambda_2}{n-1} \geq \dotsb \geq \frac{\lambda_n}{1} \geq 0. \]
The set $\LH$ of Lecture Hall partitions can be viewed as the set of lattice points in the \emph{Lecture Hall cone}, which is the cone over the simplex with vertices
\[ 
\cvec{1 \\0   \\   0\\\vdots\\ 0\\0},
\cvec{n \\ n-1\\   0\\\vdots\\ 0\\ 0},
\cvec{n \\ n-1\\ n-2\\\vdots\\ 0\\ 0},
\dotsc,
\cvec{n \\ n-1\\ n-2\\\vdots\\ 2\\ 0},
\cvec{n \\ n-1\\ n-2\\\vdots\\ 2\\ 1}.
\]
A Hilbert Basis for $\LH$ was given in \cite[Theorem 5.3]{BBKSZ}, see Remark \ref{rem:conj} below.
Lecture Hall partitions were introduced by Bousquet-M\'elou and Eriksson in \cite{BME}.
Their original motivation for considering Lecture Hall partitions is their \emph{Lecture Hall Theorem}.
\begin{theorem}[Lecture Hall Theorem, {\cite[Theorem 1.1]{BME}}]\label{thm:lht}	
	The number of Lecture Hall Partitions in $\LH$ which sum to $N \in \NN$ is equal to the number of partitions of $N$ into odd parts less than $2n$.
\end{theorem}
\noindent The generating function version of the Lecture Hall Theorem is the identity
\begin{equation}\label{eq:lh1}
	\sum_{\lambda \in \LH} q^{|\lambda|} = \prod_{i=1}^n \frac{1}{1-q^{2i-1}},
\end{equation}
where $|\lambda| := \sum_i \lambda_i$.
In this note we consider a bivariate refinement of \eqref{eq:lh1} which also due to Bousquet-M\'elou and Eriksson.
In order to state it we define
\begin{align*}
|\lambda|_o &:= \lambda_1 + \lambda_{3} + \lambda_{5} + \dotsb\\
|\lambda|_e &:= \lambda_{2} + \lambda_{4} + \lambda_{6} + \dotsb.
\end{align*}
\begin{theorem}[Lecture Hall Theorem, bivariate version {\cite[Eq. (2)]{BME}}]\label{thm:lht2}
It holds that
\begin{equation}\label{eq:lh2}
	 \sum_{\lambda \in \LH} q_1^{|\lambda|_o} q_2^{|\lambda|_e} = \prod_{i=1}^n \frac{1}{1-q_1^{i} q_2^{i-1}}.
\end{equation}
\end{theorem}

Clearly, Theorem \ref{thm:lht} follows from Theorem \ref{thm:lht2} by setting $q_1 = q_2$.
There are numerous extensions and generalizations of Theorem \ref{thm:lht} and  Theorem \ref{thm:lht2} in the literature.
We refer the reader to the survey article by Savage \cite{savagesurvey} for a wealth of references.
However, even though there are a number of proofs of Theorem \ref{thm:lht}, the result is still not considered to be well understood:
\begin{quote}
	[...], Theorem 1.2 is hardly understood at all.
	This is in spite of the fact that by now there are many proofs, including those of Bousquet-M\'elou and Eriksson [8–10], Andrews [1], Yee [55,56], Andrews, Paule, Riese, and Strehl [3], Eriksen [31], and Bradford et al.~[11].
	We have also contributed to the collection of proofs with co-authors Corteel [25], Corteel and Lee [20], Andrews and Corteel [2], Bright [15], and, most recently, Corteel and Lovejoy [23].
\end{quote}
\hfill C.D. Savage, in \cite{savagesurvey}

In this note, we propose a new approach to the Lecture Hall Theorem.
Roughly speaking, we interpret the left-hand side of \eqref{eq:lh2} as the Hilbert series of the Ehrhart ring of $\LH$ with respect to a particular grading, and the
right-hand side of \eqref{eq:lh2} as the Hilbert series of a polynomial ring.
Then, we conjecture that Ehrhart ring of $\LH$ can be obtained as a \enquote{deformation} of a polynomial ring.

To make this idea precise, we recall some background on toric deformations.
Let $\kk$ be a field, $S := \kk[x_1, \dots, x_m]$ the polynomial ring over it and $\prec$ a term order on $S$.
For a finitely generated graded sub-$\kk$-algebra $A \subseteq S$ we consider the \emph{initial subalgebra} $\IN_\prec(A) := \kk[\IN_\prec(f) \with f \in A] \subset S$ of $A$, where $\IN_\prec(f)$ denotes the leading term of $f$.
Initial subalgebras have been studied in the context of \emph{SAGBI bases} (also known as \emph{canonical bases}), and we give some background on their theory in Section \ref{sec:sagbi}.
We offer the following conjecture:
\begin{conjecture}\label{conj:vague}
	Fix $n \in \NN$.
	There exist a $\ZZ^2$-graded polynomial ring $\kk[x_1, \dots, x_m]$ for some $m \in \NN$, a graded subalgebra $A_n \subseteq \kk[x_1, \dots, x_m]$ and a term order $\prec$ on $\kk[x_1, \dots, x_m]$, such that
	\begin{enumerate}[(a)]
		\item $A_n$ is isomorphic to a polynomial ring in $n$ variables of degrees $(1,0), (2,1), \dotsc, (n,n-1)$, and
		\item $\IN_\prec(A_n)$ is isomorphic to the Ehrhart ring of $\LH$.
	\end{enumerate}
\end{conjecture}
Note that the Hilbert series of $A_n$ in the conjecture equals
\[
\prod_{i=1}^n \frac{1}{1-q_1^iq_2^{i-1}}.
\]
Moreover, the Hilbert series of $A_n$ and $\IN_\prec(A)$ coincide \cite[Proposition 2.4]{CHV}, therefore Conjecture \ref{conj:vague} implies the Lecture Hall Theorem.

We are going to give a much more precise version of Conjecture \ref{conj:vague} below as Conjecture \ref{conj:sagbi}.
For this, we are first going to define the \emph{Lecture Hall polynomials} in Section \ref{sec:lhp}. 
These are Laurent polynomials which we conjecture to be polynomials (Conjecture \ref{conj:pi}) and which generate our candidate for the algebra $A_n$ in Conjecture \ref{conj:vague}.
We also have a conjecture for a minimal SAGBI basis of $A_n$, and in Theorem \ref{thm:evidence} we give some computational evidence.


If Conjecture \ref{conj:sagbi} is true, then it gives rise to a bijection $\varphi_n$ from the set of subsets of $[n-1]$ to the Hilbert basis of $\LH$.
In Section \ref{sec:phi} we give some properties of this map, and in Table \ref{tab:phi} we list its values for $n \leq 8$. 
Finding a description of this map for all $n$ might be a problem of independent interest.

\subsection*{Acknowledgments}
The author thanks Matthias Beck for bringing the Lecture Hall Theorem to the authors attention.
Moreover, I would like thank Victor Reiner, Benjamin Braun and Volmar Welker for several helpful discussions.

Research that led to this paper was supported by the National Science Foundation under Grant No.~DMS-1440140 while the author was in residence at the Mathematical Sciences Research Institute in Berkeley, California, during the Fall 2017 semester on \emph{Geometric and Topological Combinatorics}.

\section{Notation and conventions}
For $n \in \NN$ we set $[n] := \set{1, \dotsc, n}$. For a finite set $S$ we write $2^S$ for its power set.
For an $(n \times n)$-matrix $M$ and two sets $S,T \subseteq [n]$ of the same cardinality we write $\sm{S}{T}{M}$ for the submatrix of $M$ using the rows and columns with indices in $S$ and $T$, respectively.

Most of our discussion is independent on the group field $\kk$.
Therefore, we chose to work over the rationals $\QQ$ for concreteness.

\section{The Lecture Hall polynomials}\label{sec:lhp}

In this section, we are going to define the Lecture Hall polynomals, which generate our candidate for the algebra $A_n$ of Conjecture \ref{conj:vague}.

Let $S_n := \QQ[y_1, y_2, \dotsc, y_n]$ be a polynomial ring in $n$ variables.
It is convenient to set $S_\infty := \bigcup_n S_n$, using the natural inclusion $S_{n-1} \hookrightarrow S_n$.
For a sequence of polynomials $\PP := P_1, P_2, \dotsc$ in $S_\infty$ we define an infinite matrix $M(\PP)$ by setting
\[ M(\PP)_{i,j} := 
\begin{cases}
	-P_{j-i+1} &\text{ if } j \geq i\\
	0 & \text{ otherwise. }
\end{cases}
\]
Explicitly, $M(\PP)$ looks as follows:
\[ M(\PP) = \begin{pmatrix}
-P_1   & -P_2   & -P_3   & -P_4   & \dots \\
   0   & -P_1   & -P_2   & -P_3   & \dots \\
   0   &  0     & -P_1   & -P_2   & \dots \\
   0   &  0     &  0     & -P_1   & \dots \\
\vdots & \vdots & \vdots & \vdots & \ddots
\end{pmatrix}
\]
For $i \in \NN$ let 
\[\mino_i(\PP) := -\det\sm{[\lceil i/2\rceil]}{\set{\lfloor i/2\rfloor, \dotsc, i}}{M(\PP)},\]
i.e., the negative of the minor of $M(\PP)$ which uses the $\lceil\frac{i}{2}\rceil$ many top rows and the columns with indices $\lfloor\frac{i}{2}\rfloor, \dotsc, i$.
One might think of $\mino_i(\PP)$ as the negative of the determinant the maximal top-aligned square submatrix of $M(\PP)$, whose top right corner is $-P_i$ and which does not contain any of the zeros of $M(\PP)$.

\begin{definition}
	We define the \emph{Lecture Hall sequence} to be the sequence $\LHP := \lh_1, \lh_2, \dotsc$ of rational functions in $\Quot(S_\infty)$ which are defined by requiring that 
	\begin{equation}\label{eq:lhp}\tag{$\mathrm{LHS}_i$}
		\mino_i(\LHP) = y_1^{i}y_2^{i-1}\dotsm y_{i-1}^2y_i
	\end{equation}
	for all $i \geq 1$.	We call the elements of $\LHP$ \emph{Lecture Hall polynomials} and denote them with $\lh_1, \lh_2, $ and so on.
\end{definition}
We talk about \enquote{Lecture Hall polynomials} instead of \enquote{Lecture Hall rational functions}, because they are Laurent polynomials (Proposition \ref{prop:lhp1}) and we conjecture them to be actual polynomials.

\begin{example}
	Here we compute the first Lecture Hall polynomials. 
	The equations ($\mathrm{LHS}_i$) for $i = 1,\dotsc, 4$ are:
	\begin{align*}
		y_1					&= \mino_1(\LHP) = - \det(-\lh_1) = \lh_1 \\
		y_1^2y_2			&= \mino_2(\LHP) = - \det(-\lh_2) = \lh_2 \\
		y_1^3y_2^2y_3		&= \mino_3(\LHP) = - \det\begin{pmatrix} -\lh_2 & -\lh_3 \\ -\lh_1 & -\lh_2 \end{pmatrix} = \lh_1\lh_3 - \lh_2^2 \\
		y_1^4y_2^3y_3^2y_4	&= \mino_4(\LHP) = - \det\begin{pmatrix} -\lh_3 & -\lh_4 \\ -\lh_2 & -\lh_3 \end{pmatrix} = \lh_2\lh_4 - \lh_3^2
	\end{align*}
	We read off the first two equations that $\lh_1 = y_1$ and $\lh_2 = y_1^2y_2$.
	Using this we can solve the third one for $\lh_3$ and obtain that $\lh_3 = y_1^3y_2^2 + y_1^2y_2^2y_3 = y_1^2y_2^2(y_1+y_3)$.
	This in turn allows us to solve the fourth equation for $\lh_4$, which yields that
	\[ \lh_4 = y_1^2y_2^3(y_1+y_3)^2 + y_4^2y_3^2y_2^2y_1.\]
\end{example}

\begin{proposition}\label{prop:lhp1}\leavevmode
	\begin{enumerate}[(a)]
		\item The Lecture Hall sequence is well-defined.
		\item Each $\lh_i$ is a Laurent polynomial and has coefficients in $\ZZ$.
		\item For each $i \geq 0$, the $i$-th Lecture Hall polynomial $\lh_i$ depends only on the variables $y_1, \dotsc, y_i$, and it is non-constant as a function of $y_i$.
	\end{enumerate}
\end{proposition}
\begin{proof}
	Let $\PP := P_1, P_2, \dotsc$ be a sequence of polynomials in $S_\infty$.
	We note that $\mino_i(\PP)$ depends only on $P_1, \dotsc, P_i$, and it is linear in $P_i$.
	Therefore, \eqref{eq:lhp} implies that each $P_i$ is a rational function in the $P_j$ for $j < i$, and we can solve the equations \eqref{eq:lhp} in an iterative way for each $i$.
	Thus, $\LHP$ is well-defined.
	
	For the second item, note that for $i \geq 3$, the coefficient of $P_i$ in \eqref{eq:lhp} equals the determinant of $\sm{\set{2,\dotsc, \lceil i/2\rceil}}{\set{\lfloor i/2\rfloor, \dotsc, i-1}}{M(\LHP)}$ up to a sign.
	Since the entries of $M(\LHP)$ are constant along diagonals, the latter equals $\mino_{i-2}(\LHP)$. 
	By construction, this minor is a monomial, and thus solving for $\lh_i$ yields a Laurent polynomial with coefficients in $\ZZ$.
	
	For the last item, note that it follows from our discussion that $\lh_i$ is determined by ($\mathrm{LHS}_1$), ($\mathrm{LHS}_2$), $\dotsc$, ($\mathrm{LHS}_i$).
	But those equations involve only the variables $y_1, \dotsc, y_i$ and thus $\lh_i$ depends only on them.
	Further, the right-hand side of ($\mathrm{LHS}_i$) is non-constant as a function of $y_i$, and thus the same holds for the left-hand side.
	But the left-hand side is a polynomial in the $\lh_1, \dotsc, \lh_i$, and all $\lh_j$ for $j < i$ do not depend on $y_i$.
	Thus $\lh_i$ cannot be constant as a function of $y_i$.
\end{proof}

Based on computational evidence, we offer the following conjecture:
\begin{conjecture}\label{conj:pi}
	Each Lecture Hall polynomial $\lh_i$ is a polynomial.
\end{conjecture}
We verified this conjecture for $i \leq 12$.
A list of the $\lh_i$ for $i \leq 8$ is given below in Section~\ref{sec:appPoly}.
%
In view of our application to the Lecture Hall cone, We define a $\ZZ^2$-grading on $S_\infty$ by setting
\[ 
\deg y_i := \begin{cases}
	(0,1) &\text{if $i$ is even,}\\
	(1,0) &\text{if $i$ is odd.}\\
\end{cases}
\]
\begin{proposition}\label{prop:lhp2}
	Each Lecture Hall polynomial $\lh_i$ is homogeneous of degree $(i, i-1)$ with respect to the given $\ZZ^2$-grading on $S_\infty$.
\end{proposition}
\begin{proof}
	This is a tedious but straightforward computation. We omit the details.
\end{proof}

The Propositions \ref{prop:lhp1} and \ref{prop:lhp2} together imply the following corollary:
\begin{corollary}\label{cor:hilbert}
	The algebra $A_n := \QQ[\lh_1, \dotsc, \lh_n] \subset \Quot(S_\infty)$ generated by the first $n$ Lecture Hall polynomials is isomorphic to a $\ZZ^2$-graded polynomial ring.
	Its Hilbert series equals
	\[ \prod_{i=1}^n \frac{1}{1-q_1^i q_2^{i-1}} \]
\end{corollary}
\begin{proof}
	Part (c) of Proposition \ref{prop:lhp1} implies that all $\lh_i$ are algebraically independent, and thus $A_n$ is a polynomial ring.
	The claim for the Hilbert series is then immediate from Proposition \ref{prop:lhp2}.
\end{proof}
Since each $\lh_i$ depends only on $y_1, \dotsc, y_i$ and is a Laurent polynomial, the algebra $A_n$ is actually a subalgebra of $\QQ[y_1^\pm, \dotsc, y_n^\pm]$.
Moreover, if Conjecture \ref{conj:pi} holds, then we even have that $A_n \subseteq \QQ[y_1, \dotsc, y_n]$.

\section{Realizing the Ehrhart ring as an initial subalgebra}\label{sec:sagbi}

Let us return to a general discussion of initial subalgebras.
Let $A \subseteq S := \QQ[x_1, \dots, x_m]$ be a subalgebra.
A finite collection of polynomials $p_1, \dots, p_r \in A$ is called a SAGBI basis (Subalgebra Analogue to Gr\"obner Basis for Ideals) if $\IN_\prec(A)$ is generated by $\IN_\prec(p_1), \dotsc, \IN_\prec(p_r)$ as $\QQ$-algebra.
SAGBI bases were introduced by Robbiano and Sweedler \cite{RS}, and independently by Kapur and Madeler \cite{KM}.
Their theory is in many ways similar to the theory of Gr\"obner bases, with the important difference that not every finitely-generated subalgebra admits a finite SAGBI basis.
We refer the reader to Chapter 11 of \cite{Stu96} and to the short survey by Bravo \cite{bravosurvey} for more information about these concepts.

We now introduce our candidate for a SAGBI basis for the algebra generated by the Lecture Hall polynomials.
\begin{definition}
	For a finite set $S \subseteq \NN, S \neq \emptyset$ let 
	\[ \lh_S := - \det \sm{[\#S]}{S+1}{M(\LHP)}, \]
		where $S+1 := \set{s+1 \with s \in S}$.
	In other words, $\lh_S$ is the negative of the minor of $M(\LHP)$ using $\#S$ many top rows and the columns in $S+1$.
	In addition, we set $\lh_\emptyset := \lh_1$.
\end{definition}
\noindent Note that $\lh_i = \lh_{\set{i-1}}$ for $i \in \NN$, and that $\lh_{\set{\lfloor i/2 \rfloor,\dotsc,i-1}} = \mino_i(\LHP) = y_1^{i}y_2^{i-1}\dotsm y_{i-1}^2y_i$.
Now we can state the precise version of Conjecture \ref{conj:vague}:
\begin{conjecture}\label{conj:sagbi}
	Let $\prec$ be the degree-lexicographic term order on $S_n := \QQ[y_1, \dotsc, y_n]$ with $y_1 \succ y_2 \succ \dotsc \succ y_n$.
	Assume that Conjecture \ref{conj:pi} holds, i.e., that the $\lh_i$ are polynomials. Further, let $A_n =\QQ[\lh_1, \dotsc, \lh_n]$.
	Then:
	\begin{enumerate}[(i)]
		\item The initial subalgebra $\IN_\prec(A_n)$ equals the Ehrhart ring of $\LH$.
		\item For each $S \subseteq [n-1]$, the leading term of $\lh_S$ is a monomial whose exponent vector is a Lecture Hall partition, and the Lecture Hall partitions arising in this way form the Hilbert basis of $\LH$.
		In particular, the set $\set{\lh_S \with S \subseteq [n-1]}$ is a SAGBI basis for $A_n$.
	\end{enumerate}
\end{conjecture}

\begin{remark}\label{rem:conj}
\begin{enumerate}[(a)]
	\item Part (i) of Conjecture \ref{conj:sagbi} implies the Lecture Hall Theorem.
	\item Our conjectured SAGBI basis has $2^{n-1}$ elements, and this is also the cardinality of the Hilbert basis $H_n$ of $\LH$.
	Indeed, by \cite[Theorem 5.3]{BBKSZ}, the Hilbert basis $H_n$ consists of the vectors of the form
	\begin{equation}\label{eq:hb}
		(a_1, a_2, \dotsc,  a_{i-1}, a_{i}, 0, \dotsc, 0)
	\end{equation}
	with $a_1,\dotsc a_{i} \in \ZZ$, $a_2 > a_3 > \dotsb > a_i > 0$ and $a_1 = a_{2} + 1$.
	Omitting the first coordinate yields a bijection from $H_n$ to $2^{[n-1]}$, and thus $\#H_n = 2^{n-1}$.

	\item Part (ii) of Conjecture \ref{conj:sagbi} alone implies that the Ehrhart ring of $\LH$ is contained in $\IN_\prec(A_n)$.
	Hence part (ii), Corollary \ref{cor:hilbert} and Lecture Hall Theorem together imply part (i).
	Moreover, since the cardinality of $H_n$ and of $\set{\lh_S \with S \subseteq [n-1]}$ are the same, Conjecture \ref{conj:sagbi} follows once one can show the following:
	\begin{enumerate}[(i)]
		\item For each $S \subseteq [n-1]$, the exponent vector of the leading term of $\lh_S$ is of the form \eqref{eq:hb}, and
		\item for any two distinct sets $S, S' \subseteq [n-1], S \neq S'$, the exponent vectors of the leading terms of $\lh_S$ and $\lh_{S'}$ are different.
	\end{enumerate}
	On the other hand, it seems desirable to find a proof of Conjecture \ref{conj:sagbi} which does not rely on the Lecture Hall Theorem.
\end{enumerate}
\end{remark}

\noindent Let us verify the conjecture in two small cases.

\begin{example}
	The cases $n = 1$ and $n = 2$ are rather trivial, so we consider the case $n = 3$.
	We have that $\IN_\prec(\lh_\emptyset) = \IN_\prec(\lh_1) = y_1$, $\IN_\prec(\lh_{\set{1}}) = \IN_\prec(\lh_2) = y_1^2 y_2$ and $\IN_\prec(\lh_{\set{2}}) = \IN_\prec(\lh_3) = y_1^3 y_2^2$.
	Moreover, by definition we have that $\lh_{1,2} = - \mino_3(\LHP) = y_1^3 y_2^2 y_3$.
	Hence the criterion of part (c) of Remark \ref{rem:conj} is satisfied.
\end{example}

\begin{example}
	Next we consider the case $n = 4$.
	In addition to the computations above, we have that $\IN_\prec(\lh_{\set{3}}) = \IN_\prec(\lh_4) = y_1^4 y_2^3$ and $\lh_{\set{2,3}} = - \mino_4(\LHP) = y_1^4 y_2^3 y_3^2 y_4$.
	For the remaining two polynomials we compute that
	\begin{dmath*}
		\lh_{\set{1,3}} = - \det\begin{pmatrix} -\lh_2 & - \lh_ 4 \\ - \lh_1 & -\lh_3 \end{pmatrix} = \lh_1 \lh_4 - \lh_2 \lh_3 = y_1^3 y_2^3 (y_1 + y_3)^2 + y_1^3 y_2^2 y_3^2 y_4 - y_1^4 y_2^3 (y_1 + y_3) = y_1^4 y_2^3 y_3 + y_1^3 y_2^3 y_3^2 + y_1^3 y_2^2 y_3^2 y_4,
	\end{dmath*}
	and thus $\IN_\prec(\lh_{\set{1,3}}) = y_1^4 y_2^3 y_3$.
	Moreover, 
	\begin{dmath*}
		\lh_{\set{1,2,3}} = - \det\begin{pmatrix} -\lh_2 & -\lh_3 & - \lh_ 4 \\ - \lh_1 & -\lh_2 & - \lh_3 \\ 0 & - \lh_1 & -\lh_2 \end{pmatrix} 
		= \lh_2 \det\begin{pmatrix} -\lh_2 & - \lh_ 3 \\ - \lh_1 & -\lh_2 \end{pmatrix} 
			- \lh_1 \det\begin{pmatrix} -\lh_3 & - \lh_ 4 \\ - \lh_1 & - \lh_2 \end{pmatrix}
		= - \lh_2 y_1^3y_2^2y_3 + \lh_1( y_1^4 y_2^3 y_3 + y_1^3 y_2^3 y_3^2 + y_1^3 y_2^2 y_3^2 y_4) = y_1^4 y_2^3 y_3^2 + y_1^4 y_2^2 y_3^2 y_4
	\end{dmath*}
	which implies that $\IN_\prec(\lh_{\set{1,2,3}}) = y_1^4 y_2^3 y_3^2$.
	In conclusion, the criterion of part (c) of Remark \ref{rem:conj} is satisfied.
\end{example}

\noindent As in these examples, the items in part (c) of Remark \ref{rem:conj} can be verified computationally for small $n$.
Using \texttt{Maple}, we found the following computational evidence for Conjecture \ref{conj:sagbi}:
\begin{theorem}\label{thm:evidence}
	Conjecture \ref{conj:sagbi} holds for $n \leq 12$.
\end{theorem}

\begin{remark}
	In all examples we computed so far, the leading term of $\lh_S$ has coefficient $1$.
	This is the reason for our choice of signs in the definition of the Lecture Hall polynomials.
\end{remark}

\begin{remark}
	Conjecture \ref{conj:sagbi} implies a particular description of the toric ideal associated to $\LH$.
	Consider the polynomial ring $R_n := \QQ[x_S \with S \subseteq [n-1]]$ and let $I_n \subseteq R_n$ be the kernel of the natural map $R_n \to A_n, x_S \mapsto \lh_S$.
	Since the $\lh_i$ are algebraically independent for $i \geq 1$, it is not difficult to see that $I_n$ is generated by the defining relations of the $\lh_S$, i.e., by equations of the form 
	$x_S + \det \sm{[\#S]}{S+1}{M(x_\emptyset, x_{\set{1}}, \dotsc)}$.
	
	If Conjecture \ref{conj:sagbi} is true, then by Theorem 11.4 of \cite{Stu96} the defining ideal of $\IN_\prec(A_n)$, (i.e., the toric ideal of $\LHP$) is an initial ideal of $I_n$ with respect to a particular non-generic weight order. 
	Therefore one can compute the toric ideal by computing a Gr\"obner basis of $I_n$ with respect to that term order.
	
	We computed this Gr\"obner basis using \texttt{Macaulay2} for $n \leq 7$, and, as expected from Theorem \ref{thm:evidence}, its initial forms generate the toric ideal.
	We would like to point out that $I_7$ has $2^{7-1} - 7 = 53$ generators, while the toric ideal and the Gr\"obner basis have $1351$ generators each.
	Therefore, these considerations might be helpful in understanding the toric ideals of $\LH$.
\end{remark}

\section{A certain bijection on finite sets}\label{sec:phi}

For a subset $S \subseteq [n-1]$ let $\varphi_n(S)$ be the exponent vector of $\IN_\prec(\lh_S)$.
If Conjecture \ref{conj:sagbi} is true, then $\varphi_n(S)$ a bijection between $2^{[n-1]}$ and the Hilbert basis $H_n$ of $\LH$.
As mentioned above in Remark \ref{rem:conj}, $H_n$ can itself be identified with the power set of $[n-1]$, so $\varphi_n$ can also be considered as a permutation of this set.

\begin{problem}
	Find a combinatorial description of $\varphi_n$.
\end{problem}
A table of the values $\varphi_n$ for $n \leq 8$ is provided in Table \ref{tab:phi} in Section \ref{sec:appPoly}.
We caution the reader that because Conjecture \ref{conj:sagbi} is just a conjecture, a priory the map $\varphi_n$ might fail to be injective or might fail to take values in $H_n$ for large $n$, even though we do not expect this to happen.
We close this note with a few properties of $\varphi_n$.

\begin{proposition}
	Assume that Conjecture \ref{conj:sagbi} is true.
	Then the following holds for $n \in \NN$:
	\begin{enumerate}[(a)]
		\item For each $i \leq n$, it holds that $\varphi_n(2^{[i-1]}) = H_{i}$.
		\item For $S \subseteq [n-1], S\neq \emptyset$, the alternating sum of the entries of $\varphi_n(S)$ equals $\#S$.
		\item For $S \subseteq [n-1], S\neq \emptyset$, the sum of the entries of $\varphi_n(S)$ equals $2\sum_{s \in S}s + 2\#S - \#S^2$.
	\end{enumerate}
\end{proposition}
\begin{proof}
	It follows from the construction that the restriction of $\varphi_n$ to $2^{[i-1]}$ equals $\varphi_i$, and thus Conjecture \ref{conj:sagbi} implies the first claim.
	
	The second and third claim follow from considering the $\ZZ^2$-grading on $S_n$.
	Let $r := \#S$ and $S = \set{s_1, \dotsc, s_r}$ with $s_1 < s_2 < \dotsb < s_R$.
	Expanding the minor from the definition of $\lh_S$, one sees that one of the summands is $\prod_{i=1}^r \lh_{s_i+2-i}$, and hence $\deg \lh_S = \deg \prod_{i=1}^r \lh_{s_i+2-i}$.
	The sum and the difference between the two components of the degree correspond to the sum and the alternating sum of the exponent vector of $\IN_\prec(\lh_S)$.
	From this the claimed formulas follow from a straight-forward computation.
\end{proof}


\section{Some computational data}
\label{sec:appPoly}
\noindent Here is a list of the Lecture Hall polynomials $\lh_i$ for $i =1, \dotsc, 8$.
\begin{dgroup*}
\begin{dmath*}
\lh_1 = y_1
\end{dmath*}

\begin{dmath*}
\lh_2 = y_1^2 y_2
\end{dmath*}

\begin{dmath*}
\lh_3 = y_1^2 y_2^2(y_1 + y_3)
\end{dmath*}

\begin{dmath*}
\lh_4 = y_1^2 y_2^3(y_1 + y_3)^2 + y_1^2 y_2^2 y_3^2 y_4
\end{dmath*}

\begin{dmath*}
\lh_5 = y_1^2 y_2^4(y_1 + y_3)^3
	+ 2 y_1^2 y_2^3 y_3^2 y_4 (y_1 + y_3)
	+ y_1^2 y_2^2 y_3^2 y_4^2 (y_3 - y_5)
\end{dmath*}

\begin{dmath*}
\lh_6 = y_1^2 y_2^5(y_1 + y_3)^4
	+ 3 y_1^2 y_2^4 y_3^2 y_4 (y_1 + y_3)^2
	+ y_1^2 y_2^2 y_3^2 y_4^3 (y_3 - y_5)^2
	+ y_1^2 y_2^2 y_3^2 y_4^2 \left(2 y_1 y_2 y_3 + 3 y_2 y_3^2  - 2 y_1 y_2 y_5 - 2 y_2 y_3 y_5 - y_5^2 y_6\right)
\end{dmath*}

\begin{dmath*}
\lh_7 = y_1^2 y_2^6 (y_1 + y_3)^5
	+ 4 y_1^2 y_2^5 y_3^2 y_4 (y_1 + y_3)^3
	+ y_1^2 y_2^2 y_3^2 y_4^4 (y_3 - y_5)^3
	+ y_1^2 y_2^2 y_3^2 y_4^2 \left(3 y_1^2 y_2^2 y_3 - 3 y_1^2 y_2^2 y_5 + 9 y_1 y_2^2 y_3^2 - 6 y_1 y_2^2 y_3 y_5 + 2 y_1 y_2 y_3^2 y_4 - 4 y_1 y_2 y_3 y_4 y_5 + 2 y_1 y_2 y_4 y_5^2 - 2 y_1 y_2 y_5^2 y_6 + 6 y_2^2 y_3^3 - 3 y_2^2 y_3^2 y_5 + 4 y_2 y_3^3 y_4 - 6 y_2 y_3^2 y_4 y_5 + 2 y_2 y_3 y_4 y_5^2 - 2 y_2 y_3 y_5^2 y_6 - 2 y_3 y_4 y_5^2 y_6 + 2 y_4 y_5^3 y_6 - y_5^3 y_6^2 + y_5^2 y_6^2 y_7\right)
\end{dmath*}

\begin{dmath*}
\lh_8 = y_1^2 y_2^7 (y_1 + y_3)^6 
	+ 5 y_1^2 y_2^6 y_3^2 y_4 (y_1 + y_3)^4 
	+ y_1^2 y_2^2 y_3^2 y_4^5 (y_3 - y_5)^4 \\
	+ y_1^2 y_2^2 y_3^2 y_4^2\left(4 y_1^3 y_2^3 y_3 - 4 y_1^3 y_2^3 y_5 + 18 y_1^2 y_2^3 y_3^2 - 12 y_1^2 y_2^3 y_3 y_5 + 3 y_1^2 y_2^2 y_3^2 y_4 - 6 y_1^2 y_2^2 y_3 y_4 y_5 + 3 y_1^2 y_2^2 y_4 y_5^2 - 3 y_1^2 y_2^2 y_5^2 y_6 + 24 y_1 y_2^3 y_3^3 - 12 y_1 y_2^3 y_3^2 y_5 + 12 y_1 y_2^2 y_3^3 y_4 - 18 y_1 y_2^2 y_3^2 y_4 y_5 + 6 y_1 y_2^2 y_3 y_4 y_5^2 - 6 y_1 y_2^2 y_3 y_5^2 y_6 + 2 y_1 y_2 y_3^3 y_4^2 - 6 y_1 y_2 y_3^2 y_4^2 y_5 + 6 y_1 y_2 y_3 y_4^2 y_5^2 - 4 y_1 y_2 y_3 y_4 y_5^2 y_6 - 2 y_1 y_2 y_4^2 y_5^3 + 4 y_1 y_2 y_4 y_5^3 y_6 - 2 y_1 y_2 y_5^3 y_6^2 + 2 y_1 y_2 y_5^2 y_6^2 y_7 + 10 y_2^3 y_3^4 - 4 y_2^3 y_3^3 y_5 + 10 y_2^2 y_3^4 y_4 - 12 y_2^2 y_3^3 y_4 y_5 + 3 y_2^2 y_3^2 y_4 y_5^2 - 3 y_2^2 y_3^2 y_5^2 y_6 + 5 y_2 y_3^4 y_4^2 - 12 y_2 y_3^3 y_4^2 y_5 + 9 y_2 y_3^2 y_4^2 y_5^2 - 6 y_2 y_3^2 y_4 y_5^2 y_6 - 2 y_2 y_3 y_4^2 y_5^3 + 4 y_2 y_3 y_4 y_5^3 y_6 - 2 y_2 y_3 y_5^3 y_6^2 + 2 y_2 y_3 y_5^2 y_6^2 y_7 - 3 y_3^2 y_4^2 y_5^2 y_6 + 6 y_3 y_4^2 y_5^3 y_6 - 2 y_3 y_4 y_5^3 y_6^2 + 2 y_3 y_4 y_5^2 y_6^2 y_7 - 3 y_4^2 y_5^4 y_6 + 3 y_4 y_5^4 y_6^2 - 2 y_4 y_5^3 y_6^2 y_7 - y_5^4 y_6^3 + 2 y_5^3 y_6^3 y_7 - y_5^2 y_6^3 y_7^2 + y_5^2 y_6^2 y_7^2 y_8\right)
\end{dmath*}

\end{dgroup*}

\newcommand{\dist}{1em}
\begin{landscape}
\small \footnotesize 
\begin{table}[t]
		\begin{tabular}{r | l}
			$S$ 				 & $\varphi_n(S)$       \\ \hline\hline
			$\emptyset$ 		 & $(1)$               \\ \hline
			$\set{1}$			 &	 $(2, 1)$ \\ \hline
			$\set{2}$			 &	 $(3, 2)$ \\ \hline
			$\set{1, 2}$		 &	 $(3, 2, 1)$ \\ \hline
			$\set{3}$	 		 &	 $(4, 3)$ \\ \hline
			$\set{1, 3}$		 &	 $(4, 3, 1)$ \\ \hline
			$\set{2, 3}$		 &	 $(4, 3, 2, 1)$ \\ \hline
			$\set{1, 2, 3}$		 &	 $(4, 3, 2)$ \\ \hline
			$\set{4}$			 &	 $(5, 4)$ \\ \hline
			$\set{1, 4}$		 &	 $(5, 4, 1)$ \\ \hline
			$\set{2, 4}$		 &	 $(5, 4, 2, 1)$ \\ \hline
			$\set{1, 2, 4}$		 &	 $(5, 4, 2)$ \\ \hline
			$\set{3, 4}$		 &	 $(5, 4, 3, 2)$ \\ \hline
			$\set{1, 3, 4}$		 &	 $(5, 4, 3, 1)$ \\ \hline
			$\set{2, 3, 4}$		 &	 $(5, 4, 3, 2, 1)$ \\ \hline
			$\set{1, 2, 3, 4}$	 &	 $(5, 4, 3)$ \\ \hline
			$\set{5}$ 			 &	 $(6, 5)$ \\ \hline
			$\set{1, 5}$		 &	 $(6, 5, 1)$ \\ \hline
			$\set{2, 5}$		 &	 $(6, 5, 2, 1)$ \\ \hline
			$\set{1, 2, 5}$		 &	 $(6, 5, 2)$ \\ \hline
			$\set{3, 5}$		 &	 $(6, 5, 3, 2)$ \\ \hline
			$\set{1, 3, 5}$		 &	 $(6, 5, 3, 1)$ \\ \hline
			$\set{2, 3, 5}$		 &	 $(6, 5, 3, 2, 1)$ \\ \hline
			$\set{1, 2, 3, 5}$	 &	 $(6, 5, 3)$ \\ \hline
			$\set{4, 5}$		 &	 $(6, 5, 4, 3)$ \\ \hline
			$\set{1, 4, 5}$		 &	 $(6, 5, 4, 2)$ \\ \hline
			$\set{2, 4, 5}$		 &	 $(6, 5, 4, 3, 1)$ \\ \hline
			$\set{1, 2, 4, 5}$	 &	 $(6, 5, 4, 1)$ \\ \hline
			$\set{3, 4, 5}$		 &	 $(6, 5, 4, 3, 2, 1)$ \\ \hline
			$\set{1, 3, 4, 5}$	 &	 $(6, 5, 4, 2, 1)$ \\ \hline
			$\set{2, 3, 4, 5}$	 &	 $(6, 5, 4, 3, 2)$ \\ \hline
			$\set{1, 2, 3, 4, 5}$	 &	 $(6, 5, 4)$ \\ \hline
		\end{tabular}\hspace{\dist}
		\begin{tabular}{r | l}
			$S$					 & $\varphi_n(S)$          \\
			\hline\hline
			$\set{6}$			 &	 $(7, 6)$ \\ \hline
			$\set{1, 6}$		 &	 $(7, 6, 1)$ \\ \hline
			$\set{2, 6}$		 &	 $(7, 6, 2, 1)$ \\ \hline
			$\set{1, 2, 6}$		 &	 $(7, 6, 2)$ \\ \hline
			$\set{3, 6}$		 &	 $(7, 6, 3, 2)$ \\ \hline
			$\set{1, 3, 6}$		 &	 $(7, 6, 3, 1)$ \\ \hline
			$\set{2, 3, 6}$		 &	 $(7, 6, 3, 2, 1)$ \\ \hline
			$\set{1, 2, 3, 6}$	 &	 $(7, 6, 3)$ \\ \hline
			$\set{4, 6}$		 &	 $(7, 6, 4, 3)$ \\ \hline
			$\set{1, 4, 6}$		 &	 $(7, 6, 4, 2)$ \\ \hline
			$\set{2, 4, 6}$		 &	 $(7, 6, 4, 3, 1)$ \\ \hline
			$\set{1, 2, 4, 6}$	 &	 $(7, 6, 4, 1)$ \\ \hline
			$\set{3, 4, 6}$		 &	 $(7, 6, 4, 3, 2, 1)$ \\ \hline
			$\set{1, 3, 4, 6}$	 &	 $(7, 6, 4, 2, 1)$ \\ \hline
			$\set{2, 3, 4, 6}$	 &	 $(7, 6, 4, 3, 2)$ \\ \hline
			$\set{1, 2, 3, 4, 6}$	 &	 $(7, 6, 4)$ \\ \hline
			$\set{5, 6}$		 &	 $(7, 6, 5, 4)$ \\ \hline
			$\set{1, 5, 6}$		 &	 $(7, 6, 5, 3)$ \\ \hline
			$\set{2, 5, 6}$		 &	 $(7, 6, 5, 4, 1)$ \\ \hline
			$\set{1, 2, 5, 6}$	 &	 $(7, 6, 5, 2)$ \\ \hline
			$\set{3, 5, 6}$		 &	 $(7, 6, 5, 4, 2, 1)$ \\ \hline
			$\set{1, 3, 5, 6}$	 &	 $(7, 6, 5, 3, 1)$ \\ \hline
			$\set{2, 3, 5, 6}$	 &	 $(7, 6, 5, 4, 2)$ \\ \hline
			$\set{1, 2, 3, 5, 6}$	 &	 $(7, 6, 5, 1)$ \\ \hline
			$\set{4, 5, 6}$	 	 &	 $(7, 6, 5, 4, 3, 2)$ \\ \hline
			$\set{1, 4, 5, 6}$	 &	 $(7, 6, 5, 3, 2, 1)$ \\ \hline
			$\set{2, 4, 5, 6}$	 &	 $(7, 6, 5, 4, 3, 1)$ \\ \hline
			$\set{1, 2, 4, 5, 6}$	 &	 $(7, 6, 5, 2, 1)$ \\ \hline
			$\set{3, 4, 5, 6}$	 &	 $(7, 6, 5, 4, 3, 2, 1)$ \\ \hline
			$\set{1, 3, 4, 5, 6}$	 &	 $(7, 6, 5, 3, 2)$ \\ \hline
			$\set{2, 3, 4, 5, 6}$	 &	 $(7, 6, 5, 4, 3)$ \\ \hline
			$\set{1, 2, 3, 4, 5, 6}$ &	 $(7, 6, 5)$ \\ \hline
		\end{tabular}\hspace{\dist}
		\begin{tabular}{r | l}
			$S$					 & $\varphi_n(S)$          \\
			\hline\hline
			$\set{7}$			 &	 $(8, 7)$ \\ \hline
			$\set{1, 7}$		 &	 $(8, 7, 1)$ \\ \hline
			$\set{2, 7}$		 &	 $(8, 7, 2, 1)$ \\ \hline
			$\set{1, 2, 7}$		 &	 $(8, 7, 2)$ \\ \hline
			$\set{3, 7}$		 &	 $(8, 7, 3, 2)$ \\ \hline
			$\set{1, 3, 7}$		 &	 $(8, 7, 3, 1)$ \\ \hline
			$\set{2, 3, 7}$		 &	 $(8, 7, 3, 2, 1)$ \\ \hline
			$\set{1, 2, 3, 7}$	 &	 $(8, 7, 3)$ \\ \hline
			$\set{4, 7}$		 &	 $(8, 7, 4, 3)$ \\ \hline
			$\set{1, 4, 7}$		 &	 $(8, 7, 4, 2)$ \\ \hline
			$\set{2, 4, 7}$		 &	 $(8, 7, 4, 3, 1)$ \\ \hline
			$\set{1, 2, 4, 7}$	 &	 $(8, 7, 4, 1)$ \\ \hline
			$\set{3, 4, 7}$		 &	 $(8, 7, 4, 3, 2, 1)$ \\ \hline
			$\set{1, 3, 4, 7}$	 &	 $(8, 7, 4, 2, 1)$ \\ \hline
			$\set{2, 3, 4, 7}$	 &	 $(8, 7, 4, 3, 2)$ \\ \hline
			$\set{1, 2, 3, 4, 7}$&	 $(8, 7, 4)$ \\ \hline
			$\set{5, 7}$		 &	 $(8, 7, 5, 4)$ \\ \hline
			$\set{1, 5, 7}$		 &	 $(8, 7, 5, 3)$ \\ \hline
			$\set{2, 5, 7}$		 &	 $(8, 7, 5, 4, 1)$ \\ \hline
			$\set{1, 2, 5, 7}$	 &	 $(8, 7, 5, 2)$ \\ \hline
			$\set{3, 5, 7}$		 &	 $(8, 7, 5, 4, 2, 1)$ \\ \hline
			$\set{1, 3, 5, 7}$	 &	 $(8, 7, 5, 3, 1)$ \\ \hline
			$\set{2, 3, 5, 7}$	 &	 $(8, 7, 5, 4, 2)$ \\ \hline
			$\set{1, 2, 3, 5, 7}$	 &	 $(8, 7, 5, 1)$ \\ \hline
			$\set{4, 5, 7}$	 	 &	 $(8, 7, 5, 4, 3, 2)$ \\ \hline
			$\set{1, 4, 5, 7}$	 &	 $(8, 7, 5, 3, 2, 1)$ \\ \hline
			$\set{2, 4, 5, 7}$	 &	 $(8, 7, 5, 4, 3, 1)$ \\ \hline
			$\set{1, 2, 4, 5, 7}$	 &	 $(8, 7, 5, 2, 1)$ \\ \hline
			$\set{3, 4, 5, 7}$	 &	 $(8, 7, 5, 4, 3, 2, 1)$ \\ \hline
			$\set{1, 3, 4, 5, 7}$	 &	 $(8, 7, 5, 3, 2)$ \\ \hline
			$\set{2, 3, 4, 5, 7}$	 &	 $(8, 7, 5, 4, 3)$ \\ \hline
			$\set{1, 2, 3, 4, 5, 7}$ &	 $(8, 7, 5)$ \\ \hline
		\end{tabular}\hspace{\dist}
		\begin{tabular}{r | l}
			$S$						 & $\varphi_n(S)$          \\
			\hline\hline
			$\set{6, 7}$			 &	 $(8, 7, 6, 5)$ \\ \hline
			$\set{1, 6, 7}$			 &	 $(8, 7, 6, 4)$ \\ \hline
			$\set{2, 6, 7}$			 &	 $(8, 7, 6, 5, 1)$ \\ \hline
			$\set{1, 2, 6, 7}$		 &	 $(8, 7, 6, 3)$ \\ \hline
			$\set{3, 6, 7}$			 &	 $(8, 7, 6, 5, 2, 1)$ \\ \hline
			$\set{1, 3, 6, 7}$		 &	 $(8, 7, 6, 4, 1)$ \\ \hline
			$\set{2, 3, 6, 7}$		 &	 $(8, 7, 6, 5, 2)$ \\ \hline
			$\set{1, 2, 3, 6, 7}$	 &	 $(8, 7, 6, 2)$ \\ \hline
			$\set{4, 6, 7}$			 &	 $(8, 7, 6, 5, 3, 2)$ \\ \hline
			$\set{1, 4, 6, 7}$		 &	 $(8, 7, 6, 4, 2, 1)$ \\ \hline
			$\set{2, 4, 6, 7}$		 &	 $(8, 7, 6, 5, 3, 1)$ \\ \hline
			$\set{1, 2, 4, 6, 7}$	 &	 $(8, 7, 6, 3, 1)$ \\ \hline
			$\set{3, 4, 6, 7}$		 &	 $(8, 7, 6, 5, 3, 2, 1)$ \\ \hline
			$\set{1, 3, 4, 6, 7}$	 &	 $(8, 7, 6, 4, 2)$ \\ \hline
			$\set{2, 3, 4, 6, 7}$	 &	 $(8, 7, 6, 5, 3)$ \\ \hline
			$\set{1, 2, 3, 4, 6, 7}$ &	 $(8, 7, 6, 1)$ \\ \hline
			$\set{5, 6, 7}$			 &	 $(8, 7, 6, 5, 4, 3)$ \\ \hline
			$\set{1, 5, 6, 7}$		 &	 $(8, 7, 6, 4, 3, 2)$ \\ \hline
			$\set{2, 5, 6, 7}$		 &	 $(8, 7, 6, 5, 4, 2)$ \\ \hline
			$\set{1, 2, 5, 6, 7}$	 &	 $(8, 7, 6, 3, 2, 1)$ \\ \hline
			$\set{3, 5, 6, 7}$		 &	 $(8, 7, 6, 5, 4, 3, 1)$ \\ \hline
			$\set{1, 3, 5, 6, 7}$	 &	 $(8, 7, 6, 4, 3, 1)$ \\ \hline
			$\set{2, 3, 5, 6, 7}$	 &	 $(8, 7, 6, 5, 4, 1)$ \\ \hline
			$\set{1, 2, 3, 5, 6, 7}$ &	 $(8, 7, 6, 2, 1)$ \\ \hline
			$\set{4, 5, 6, 7}$		 &	 $(8, 7, 6, 5, 4, 3, 2, 1)$ \\ \hline
			$\set{1, 4, 5, 6, 7}$	 &	 $(8, 7, 6, 4, 3, 2, 1)$ \\ \hline
			$\set{2, 4, 5, 6, 7}$	 &	 $(8, 7, 6, 5, 4, 2, 1)$ \\ \hline
			$\set{1, 2, 4, 5, 6, 7}$ &	 $(8, 7, 6, 3, 2)$ \\ \hline
			$\set{3, 4, 5, 6, 7}$	 &	 $(8, 7, 6, 5, 4, 3, 2)$ \\ \hline
			$\set{1, 3, 4, 5, 6, 7}$ &	 $(8, 7, 6, 4, 3)$ \\ \hline
			$\set{2, 3, 4, 5, 6, 7}$ &	 $(8, 7, 6, 5, 4)$ \\ \hline
			$\set{1, 2, 3, 4, 5, 6, 7}$	 &	 $(8, 7, 6)$ \\ \hline
		\end{tabular}
	\vspace*{2ex} 
	\caption{The values of $\varphi_n$ for $n \leq 8$. 
		Trailing zeros are omitted.}\label{tab:phi}
\end{table}
\end{landscape}


\begin{thebibliography}{BBKSZ16}

\bibitem[BBKSZ16]{BBKSZ}
Matthias Beck, Benjamin Braun, Matthias K\"oppe, Carla~D. Savage, and Zafeirakis Zafeirakopoulos,
  \emph{Generating functions and triangulations for lecture hall cones},
  SIAM J. Discrete Math. \textbf{30} (2016), no.~3, 1470--1479.
  doi: \href{http://dx.doi.org/10.1137/15M1036907}{10.1137/15M1036907}

\bibitem[BME97]{BME}
Mireille Bousquet-M{\'e}lou and Kimmo Eriksson, \emph{Lecture hall partitions},
  Ramanujan J. \textbf{1} (1997), no.~1, 101--111.
  doi: \href{http://dx.doi.org/10.1023/A:1009771306380}{10.1023/A:1009771306380}

\bibitem[Bra04]{bravosurvey}
Ana Bravo, \emph{Some facts about canonical subalgebra bases}, Trends in
  commutative algebra, Math. Sci. Res. Inst. Publ., vol.~51, Cambridge Univ.
  Press, Cambridge, 2004, pp.~247--254.
  doi: \href{http://dx.doi.org/10.1017/CBO9780511756382.009}{10.1017/CBO9780511756382.009}

\bibitem[CHV96]{CHV}
Aldo Conca, J\"urgen Herzog, and Giuseppe Valla, \emph{Sagbi bases with
  applications to blow-up algebras}, J. Reine Angew. Math. \textbf{474} (1996),
  113--138.
  doi: \href{http://dx.doi.org/10.1515/crll.1996.474.113}{10.1515/crll.1996.474.113}

\bibitem[KM89]{KM}
Deepak Kapur and Klaus Madlener, \emph{A completion procedure for computing a
  canonical basis for a {$k$}-subalgebra}, Computers and mathematics
  ({C}ambridge, {MA}, 1989), Springer, New York, 1989, pp.~1--11.

\bibitem[RS90]{RS}
Lorenzo Robbiano and Moss Sweedler, \emph{Subalgebra bases}, Commutative
  algebra ({S}alvador, 1988), Lecture Notes in Math., vol. 1430, Springer,
  Berlin, 1990, pp.~61--87.
  doi: \href{http://dx.doi.org/10.1007/BFb0085537}{10.1007/BFb0085537}

\bibitem[Sav16]{savagesurvey}
Carla~D. Savage, \emph{The mathematics of lecture hall partitions}, J. Combin.
  Theory Ser. A \textbf{144} (2016), 443--475.
  doi: \href{http://dx.doi.org/10.1016/j.jcta.2016.06.006}{10.1016/j.jcta.2016.06.006}

\bibitem[Stu96]{Stu96}
Bernd Sturmfels, \emph{Gr\"obner bases and convex polytopes}, University
  Lecture Series, vol.~8, American Mathematical Society, Providence, RI, 1996.
  
\end{thebibliography}


\end{document}